\documentclass{amsart}

\newtheorem{theorem}{Theorem}
\newtheorem{lemma}[theorem]{Lemma}

\numberwithin{theorem}{section} 

\newtheorem{definition}[theorem]{Definition}

\usepackage{changes}
\definechangesauthor[name={Yan}, color=blue]{YZ}
\definechangesauthor[name={Beeson}, color=purple]{MB}

\usepackage{amsfonts}
\usepackage[section]{placeins}
\usepackage{comment}
\usepackage{amsmath}
\usepackage{array}
\usepackage{xcolor}
\usepackage{varwidth}
\usepackage{verbatim}
\usepackage{graphicx}
\usepackage{hyperref}

\def\R{{\mathbb R}}

\def\A{{\mathcal A}}
 
\def\C{{\mathcal C}} 
\def\zh{{\Phi}}

\author[Beeson]{Michael Beeson}
\address{San Jos\'e State University (emeritus), and UCSC (research associate)}

\author[Zhang]{Yan X Zhang}
\address{San Jos\'e State University}

\title{Rationality of certain triangle tilings}
\date{\today}

\begin{document}

\begin{abstract}
We consider tilings of a triangle $ABC$ by congruent copies of a triangle
that has one angle equal to $120^\circ$, has non-commensurable angles
(that is, not all angles are rational multiples of $\pi$), and is not
similar to $ABC$. We prove that any such tiling has commensurable sides,
meaning that the side lengths can be taken to be integers after scaling.
\end{abstract}
\maketitle

\section{Introduction and Motivation}

A \emph{tiling} of a triangle $ABC$, or more generally of a polygon, is an expression of $ABC$ as 
the union of non-overlapping congruent copies of a smaller triangle called the \emph{tile}. 
We use the notation $(a,b,c)$ for the sides of the triangle, and $(\alpha,\beta,\gamma)$ for
the angles opposite $(a,b,c)$ respectively. 
This is far from the first paper on the subject of tilings by triangles, and we do not provide a general introduction here.  

Two recurring conditions in studying tilings are:
\begin{itemize}
    \item we say that the tile has \emph{commensurable angles} if all angles are rational multiples of $\pi$, and \emph{non-commensurable angles} otherwise.
    \item we say that the tile has \emph{commensurable sides} if the pairwise ratios of the sides are all rational (equivalently, after a suitable scaling, they can be taken to be integers), and \emph{non-commensurable sides} otherwise.
\end{itemize}
Whether these conditions apply changes what techniques are available to decide whether and how $ABC$ can be tiled by different types of tiles. 

\begin{table}[ht]
\caption{Tilings when not all angles are rational multiples of $\pi$.}
\label{table:laczkovich2}
\centering
\setlength{\extrarowheight}{.5em}
\begin{tabular}{lr}
$ABC$  &  The tile     \\
\hline
$(\alpha,\beta,\gamma)$  &  similar to $ABC$ \\
equilateral              & $\alpha = \pi/3$ \\
$(\alpha,\alpha,2\beta)$ &  $\gamma = \pi/2$ \\
$(\alpha,\alpha,\pi-2\alpha)$  & $\gamma = 2\alpha$\\
$(2\alpha,\beta,\alpha+\beta)$ & $3\alpha + 2\beta = \pi$\\
$(2\alpha,\alpha,2\beta)$ &$3\alpha + 2\beta = \pi$\\
isosceles &$3\alpha + 2\beta = \pi$\\
$(\alpha,\alpha,\pi-2\alpha)$ & $\gamma = 2\pi/3$ \\
$(\alpha,2\alpha,\pi-3\alpha)$ & $\gamma = 2\pi/3$ \\
$(\alpha, 2 \beta, 2\alpha + \beta)$ & $\gamma = 2\pi/3$\\
$(\alpha, \alpha+\beta,\alpha+2\beta)$ & $\gamma = 2\pi/3$\\
$(2\alpha, 2\beta, \alpha+\beta)$ & $\gamma = 2\pi/3$ \\
equilateral  & $\gamma = 2\pi/3$
\end{tabular}
\end{table}

The commensurable angles case is mostly settled by Theorems 5.1 and 5.3 of \cite{laczkovich1995}, so we focus on the more difficult noncommensurable angles case. In Table~\ref{table:laczkovich2}, we exhibit a list of potential tilings when the tile has noncommensurable angles, as a consequence of Laczkovich's Theorem~4.1 of \cite{laczkovich1995}. Note that a particularly prevalent case is when the tile has one angle $2\pi/3$; the possible tilings by such tiles are not yet understood. The
known examples all have commensurable sides and more
than a thousand tiles, which suggests the following theorem:
\begin{theorem} \label{theorem:main}
Suppose a triangle $ABC$ is tiled by a tile $R$ with  sides $(a,b,c)$ and noncommensurable angles $(\alpha,\beta,\gamma = 2\pi/3)$. If $ABC$ is not similar to $R$, 
then $R$ has commensurable sides.
\end{theorem}

Laczkovich's work in \cite{laczkovich2012} contains a deep exploration of tilings (of not only triangles but also convex polygons), in many of the cases proving that the tiles must have commensurable sides. 
However, the results do not immediately settle our theorem for two reasons: First, the main results in \cite{laczkovich2012} (such as Theorem 2.1) assume ``regular tilings,'' a condition that does not apply to
the $\gamma = 2\pi/3$ cases unless $ABC$ is equilateral. Second, the relevant case of a $2\pi/3$ angle is flawed at line 13 of the proof of Lemma~7.1 of \cite{laczkovich2012}, and we could not easily repair it. We devote an Appendix to explaining both the flaw and the difficulty of repairing it, and in the body of our paper supply a self-contained proof of this $\gamma = 2\pi/3$ case, with no dependencies on previous literature\footnote{In private communication, Laczokovich fixed the issue with an independent proof, to appear on the Mathematical arXiv.}. 

Here is a brief description of the proof of
Theorem~\ref{theorem:main}.
After some preliminary setup in Section~\ref{sec:preliminaries}, we  prove $a/b$ is rational in Section~\ref{sec:ab}, using methods
from  \cite{beeson-isosceles} that in turn go back to Laczkovich.%
\footnote{   
The paper \cite{beeson-isosceles} did not reference the theorem in question here, because 
it was required to prove not only that the ratios of sides are rational, but that rational
relations between them are ``witnessed' in the tiling.} 
These methods are sufficient to prove most of the cases, but leave two exceptional cases that are more difficult, namely the equilateral triangle and a triangle with angles $(2\alpha, 2\beta, \pi/3)$. In Section~\ref{sec:invariant}, we introduce 
Laczkovich's invariant $\zh (T)$ of a tiling $T$, which is computed
as a sum of $\zh (PQR)$ over the tiles $PQR$, adding a term for each directed edge
of a tile.  
The term for each directed edge is a signed length, where the sign depends on the direction of the edge.  Finally, we use this invariant to finish the remaining two cases in Section~\ref{sec:proof-main}, proving Theorem~\ref{theorem:main} in a self-contained manner. Finally, in Section~\ref{sec:conclusion}, we combine Theorem~\ref{theorem:main} with  results already in the literature to arrive at the following theorem: 
\begin{theorem}
    \label{theorem:main-2} Let triangle $ABC$ be tiled by a tile $R$ such that
    \begin{itemize}
        \item $R$ is not similar to ABC;
        \item $R$ is not a right triangle;
        \item $R$ does not have commensurable angles.
    \end{itemize}
    Then $R$ must have commensurable sides.
\end{theorem}

\section{Preliminaries}
\label{sec:preliminaries}

In this paper, we will usually assume that the tile $R$ is a triangle with side lengths $(a,b,c)$ and corresponding angles $(\alpha,\beta,\gamma)$ respectively, with $\gamma = 2\pi/3$ and non-commensurable angles (which implies that $\alpha, \beta$ are not rational multiples of $\pi$ since $\gamma$ is). In this section, we give some definitions and tools that provide control over the structure of the tiling.

\subsection{Laczkovich's graphs}

The graphs $\Gamma_a$, $\Gamma_b$, $\Gamma_c$ are defined in in \cite[p.~346]{laczkovich2012}. They track different types of segments that can be used to write down linear relations between the side lengths of tiles.   The elements of these graphs are points.   Normally
graphs have ``nodes'' and ``edges''.   We use ``link'' instead of ``edge'' with 
these graphs, to avoid confusion with tile ``edges.''
We build up the relevant terminology as follows. All 
these definitions presume a fixed tiling.
\begin{itemize}
\item a {\em directed segment} is an ordered pair of points $PQ$ such that the line segment $PQ$ is a union of edges of tiles.  (Thus $PQ$ and $QP$ are different as directed segments.)
    \item an {\em internal segment} is a line segment connecting two vertices
of the tiling that is contained in the union of the boundaries of the 
tiles, and lies in the interior of $ABC$ except possibly for its endpoints.
\item A segment is {\em terminated}  at a vertex $P$ if it has
tiles on both sides with vertices at $P$.  (In that case there may or may not 
be a continuation of that segment past $P$.)  
\item A {\em left-terminated segment}  is an internal directed segment 
$XY$ that is terminated at $X$. A {\em right-terminated segment} is terminated at $Y$.
\item A tile is {\em supported by}
$XY$ if one edge of the tile lies on $XY$.
\item The internal directed segment $XY$ is 
said to have ``all $c$'s on the left'' if\footnote{Here the concept of ``left side'' of $XY$ is a different sense of ``left'' from ``left-terminated,'' because it means the left (counterclockwise) side of an observer standing at $X$ and looking in the direction of $Y$.} the endpoints $X$ and $Y$ are
vertices of tiles supported by $XY$ and lying on the left side of $XY$,
and all tiles supported by $XY$ lying on the left of $XY$ have their 
$c$ edges on $XY$. In this case we call the left side its \emph{$c$-side} and the right side its \emph{non-$c$-side}. Similarly for ``all $c$'s on the right.'' 
\item An {\em $c/a$  segment} is a left-terminated
 internal segment $PQ$ of the tiling
supporting two tiles on opposite sides of $PQ$,
each with a vertex at $P$, one with its $c$ edge on $PQ$ and 
one with its $a$ or $b$ edge on $PQ$.  The segment is said 
to ``emanate from $P$.''   Similarly for {\em $c/b$  segment} and {\em $a/b$  segment}.
\end{itemize}

Now, we can define Laczkovich's graphs.  A graph normally has ``nodes'' and ``edges'', but 
the word ``edge'' is already in use for tile edges, so we use the word ``link'' instead.
\begin{definition} \label{definition:c/a}
$\Gamma_a$ (and $\Gamma_b$ and $\Gamma_c$ analogously) is  the directed graph whose links are 
ordered pairs $\langle P,Q\rangle$, where $PQ$ is a 
segment $PQ$ that satisfies the following:
\begin{itemize}
    \item $PQ$ is left-terminated; 
    \item  $PQ$ has all $a$'s on one side;
    \item  $Q$ is not a vertex of a tile on the 
    non-$a$ side of $PQ$; 
    \item  there is a point $R$ such that the 
    straight line $PQR$ is part of the tiling,
    and $QR$ is a non-$a$ edge of a tile.
\end{itemize}
\end{definition}
We call $P$ the \emph{tail} and $Q$ the \emph{head} of the link $PQ$.
Sometimes we use the terminology ``$a$-link''  for a link of $\Gamma_a$.
That is, the next tile on the $a$-side of $PQ$ after all the $a$-edges
does not have its $a$-edge on $PQ$ (extended). Also, note that if $PQ$ is a link in $\Gamma_a$, its head $Q$ must be a $\pi$-vertex.

The reader who is not already familiar with these graphs is invited to fix the ideas by identifying $\Gamma_a$ in the tiling shown in Fig.~\ref{figure:1215}.

\subsection{Types of Vertices}

We classify the ``types'' of vertices that can occur in a tiling
according to the number of $\alpha,\beta$, and $\gamma$ angles occurring
at the vertex.  For example, a vertex of type $(1,1,1)$ has one each
of $\alpha$, $\beta$, and $\gamma$ angles meeting there; we call that 
a {\em simple} vertex. Note that in our definition, an $a$- (or $b$- or $c$-) link must end at a simple vertex.  Table~\ref{table:vertextypes} assigns some names to the different types
of vertices. 

\begin{table}[htp]
\caption{Vertex types.}
\label{table:vertextypes}
\centering
\begin{tabular}{ l  c c}
Name  & $(\alpha,\beta,\gamma)$ & Angle sum \\
\hline
simple    & $(1,1,1)$  & $\pi$\\
star      & $(3,3,0)$ & $\pi$\\
center    & $(0,0,3)$  & $2\pi$\\
double star  & $(6,6,0)$  & $2\pi$\\
$\gamma$-star     & $(4,4,1)$ & $2\pi$ \\
double simple & $(2,2,2)$& $2\pi$
\end{tabular}
\end{table}
Call a vertex a \emph{$\pi$-vertex} (respectively, \emph{$2\pi$-vertex}) if the sum of the angles of tiles at $X$ is $\pi$ (respectively, $2\pi$); any vertex must be either a vertex $A,B,C$ of the big triangle,  a $\pi$-, or a $2\pi$-vertex. 
Since $R$ has noncommensurable angles, we must have an equal number of $\alpha$'s and $\beta$'s at a $\pi$- or $2\pi$-vertex $X$, so the number of $\gamma$'s determine the other angles: 
\begin{itemize}
    \item $\pi$-vertex: If there is a $\gamma$ angle at $X$ then the other two are
an $\alpha$ and a $\beta$, and $X$ is a simple vertex.  If there is no 
$\gamma$ angle, then since $\pi = 3 \alpha + 3 \beta$, $X$ is a star.
\item $2\pi$-vertex:  If there are 
three $\gamma$ angles at $X$, then $X$ is a center.  If there 
are two $\gamma$ angles at $X$, then $X$ is a double simple, $(2,2,2)$.
If there is just one $\gamma$ angle at $X$, then $X$ is an $\gamma$-star  $(4,4,1)$. If there are no $\gamma$ angles, then we have a double star $(6,6,0)$.
\end{itemize}
This shows that every vertex (except the vertices of the big triangle) must be one of the types shown in Table~\ref{table:vertextypes}.

\section{Rationality of \texorpdfstring{$a/b$}{a/b}}
\label{sec:ab}

The proofs in this section are already contained in \cite{beeson-isosceles},
where they are stated only for isosceles (and not equilateral) $ABC$.  Here
we state the lemmas more generally and provide self-contained proofs, i.e., 
there are no references here to \cite{beeson-isosceles}.

\begin{lemma} \label{lemma:extendlink} 
Let a polygon be
tiled by tile $(\alpha,\beta,\gamma= 2\pi/3)$. Let $PQR$ be an (internal or boundary) segment where $Q$ is a $\pi$-vertex, such that $PQ$ is a $c$-edge of some tile and $QR$ is an $a$- or $b$-edge of some tile. Then there is a $c/a$ or $c/b$ segment emanating from $Q$.
\end{lemma}

\begin{proof} 
Checking Table~\ref{table:vertextypes}, $Q$ must be simple or a star.

If $Q$ is simple, then one of the tiles has its $\gamma$ angle at $Q$, and hence no $c$ edge at $Q$. The other two have a  $c$ edge ending at $Q$.  One of those lies on $PQ$.  The
other does {\em not} lie on $QR$, by hypothesis.  Since there
is no other $c$ edge ending at $Q$, that third $c$ edge forms
either a $c/b$ segment or a $c/a$ segment.

If $Q$ is a star, then all six angles are $\alpha$ or
$\beta$.  Each of the six tiles has a $c$ edge ending at $Q$.
One lies on $PQ$ and five lie on internal segments.  Since five is odd,
one of those $c$ edges is not paired with another $c$ edge, and hence
constitutes a $c/a$ segment or a $c/b$ segment. 
\end{proof}

\begin{figure}[ht]
\centering
\includegraphics[width=0.85\linewidth]{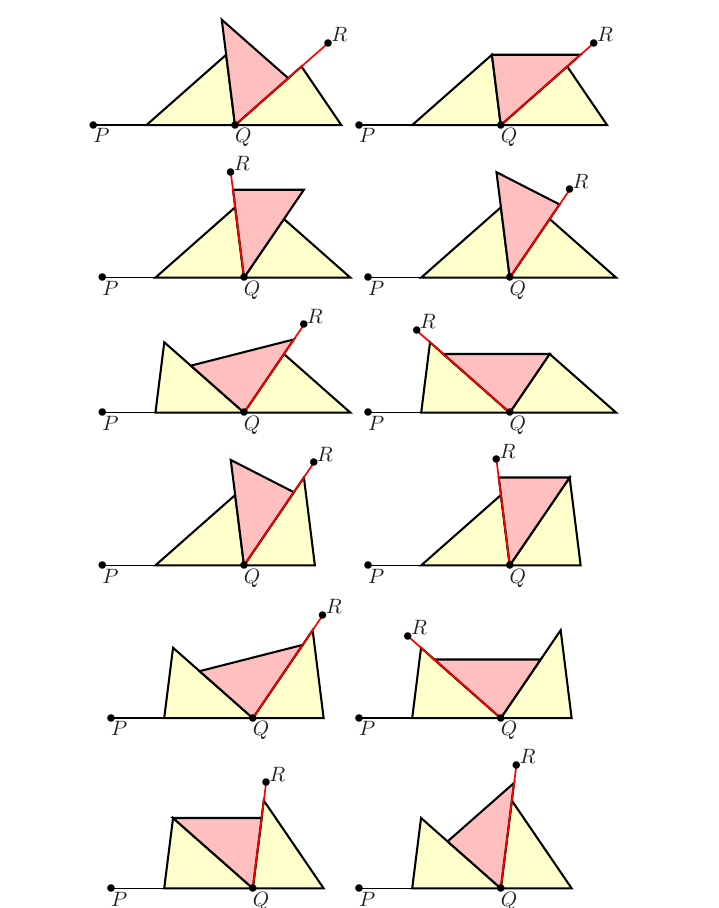}
\caption{A link $PQ$ in $\Gamma_b$ gives rise to another link through $QR$.}
\label{figure:GammaB}
\end{figure}

\begin{lemma} \label{lemma:extendlink-a} 
Let a polygon be
tiled by tile $(\alpha,\beta,\gamma= 2\pi/3)$. Let $PQR$ be an (internal or boundary) segment where $Q$ is a simple vertex, such that $PQ$ is an $a$-edge of some tile and $QR$ is an $b$- or $c$-edge of some tile. Then there is an $a/b$ or $a/c$ segment emanating from $Q$.
\end{lemma}

\noindent{\em Remark. The lemma does not seem to 
hold if $Q$ is not a simple vertex; see the Appendix.}

\begin{proof}
Identical to the simple vertex case of the proof of Lemma~\ref{lemma:extendlink}. 
Alternately, a visual proof is provided in 
Fig.~\ref{figure:GammaB}, which illustrates every case.
\end{proof}

\begin{definition} \label{definition:a-relation}
An { \bf $a$-relation} (resp. $b$-relation and $c$-relation) is a relation 
$$ja = pb + qc$$
with nonnegative integers $p,q,j$ and $j > 0$ (resp. $jb = pa + qc$ for $b$-relations and $jc = pa + qb$ for $c$-relations).
\end{definition}

\begin{lemma} \label{lemma:crelation}
Suppose a triangle $ABC$ is tiled by $(\alpha,\beta,\gamma=2\pi/3)$ with noncommensurable angles. Suppose $ABC$ is not similar
to the tile. Then either there is a $c$-relation, or $ABC$ satisfies the following conditions:
\smallskip

(i) Every tile supported by the boundary of $ABC$ has its $c$ edge on the boundary, and 

(ii) $ABC$ is either equilateral or has angles $(\pi/3, 2\alpha,2\beta)$.
\end{lemma}

\noindent{\em Remark}.  It is not asserted that the relation
is witnessed in the tiling.  Of course, if the tile is rational
there is a $c$-relation,  so we may as well assume the 
tile is not rational.
\smallskip 

\begin{proof}  Since $\gamma > \pi/2$, there is at most
one $\gamma$ angle of a tile at a vertex of triangle $ABC$. 
If there is one $\gamma$ angle, then the other two angles are
$\alpha$ and $\beta$, so $ABC$ is similar to the tile, which
is assumed not to be the case.  Therefore  there are no $\gamma$
angles of tiles at any vertex of $ABC$.

Suppose, for proof by contradiction,
that there is no relation $jc = pa + qb$.
If $PQ$ is a $c$-link
in the graph $\Gamma_c$, it must end at a $\pi$-vertex, where the continuation of the segment is supported by a non-$c$-edge. Then, by Lemma~\ref{lemma:extendlink}, there is a $c/b$ or $c/a$
segment emanating from $Q$. Extend that segment to the 
longest segment $QR$ supporting only tiles with $c$ edges
on $QR$.  Since there is no relation $jc = pa + qb$,
$R$ cannot be the vertex of a tile on the other side 
of $QR$.  Therefore $QR$ is a link in $\Gamma_c$.

Therefore the out-degree of every node $Q$ in $\Gamma_c$ is 
at least one.  But the in-degree of $\Gamma_c$ is always
at most one.  Since the total out-degree is equal to the 
total in-degree, it follows that every node of $\Gamma_c$
has both in-degree and out-degree equal to 1.  Since
no link of $\Gamma_c$ can terminate on the boundary of $ABC$,
there can be no links of $\Gamma_c$ emanating from a vertex
on the boundary of $ABC$.

We claim that there is at least one 
$c$ edge on $AC$.  For if not, every tile supported 
by $AC$ has its $\gamma$ angle at a vertex on $AC$.
Since $\gamma > \pi/2$, there cannot be two $\gamma$
angles at any vertex internal to $AC$.  This means that we must have a $\gamma$ at $A$ or a $\gamma$ at $C$, or both. However, as noted in the first
paragraph of this proof, $\gamma$ angles do not occur at $A$ or $C$. Therefore, as claimed,
there is at least one $c$ edge on $AC$.  The same argument
applies equally to the other two sides of $ABC$,  as so far
the labels of the vertices are arbitrary. 

We now claim that every tile supported by the boundary is supported by their $c$-edge. If not, then somewhere on that boundary (without loss of generality, $AC$) there must be an $a$- or $b$-edge adjacent to a $c$-edge. By Lemma~\ref{lemma:extendlink}, there would be a $c$-link emanating from this vertex, which we have just established is not possible.

Thus, the boundary is a union of only $c$-edges. Consider any vertex of the big triangle. Since it is incident to two $c$-edges, it must be incident to at least two tiles, making at least $6$ tiles total across $A$, $B$, $C$. Since their angles sum to $\pi$ and we have non-commensurable angles, we are forced to have exactly $6$ tiles contributing three each of $\alpha$ and $\beta$.  So the angles of $ABC$ must be
$(2\alpha, 2\beta, \alpha+\beta)$ or $(\alpha + \beta, \alpha+\beta,\alpha+\beta)$.
\end{proof}
 
 \begin{lemma} \label{lemma:abrelation}
Suppose a triangle $ABC$
is $N$-tiled by $(\alpha,\beta,2\pi/3)$, with $\alpha$ not a 
rational multiple of $\pi$. Suppose $ABC$ is not similar to the tile. Then there exist both an $a$-relation and a $b$-relation.
\end{lemma}

\begin{proof}
It suffices to show the first relation $ja = pb + qc$; the existence of the second comes by symmetry, since the labels
``$a$'' and ``$b$'' can be interchanged.  Suppose, for sake of contradiction,
that there is no such relation as alleged in the statement of the lemma.
Let ${\mathcal S}$ be the number of stars and $\gamma$-stars, 
${\mathcal S_2}$ the number of double stars, and $\C$ the number of centers.  Now let us calculate the 
number of $\alpha$ angles, plus the number of $\beta$ angles,
minus twice the number of $\gamma$ angles. 
What do we get at the vertices $A,B,C$?  
There is no $\gamma$ angle at any vertex of $ABC$, since if there were, the 
other angles would have to be $\alpha$ and $\beta$, making $ABC$ similar to 
the tile, which is assumed not to be the case.  So $A$, $B$, and $C$ 
together contribute three each of $\alpha$ and $\beta$.
The contributions from vertices are tabulated
by vertex type in  Table~\ref{table:vertexcount}.

\begin{table}[htp]
\caption{$\#(\alpha) + \#(\beta)-2\#(\gamma)$ by vertex type}
\label{table:vertexcount}
\centering
\begin{tabular}{ l  c }
Vertex   & $\#(\alpha) + \#(\beta)-2\#(\gamma)$  \\
\hline
simple    & $0$ \\
center    & $-6$  \\
star      & $6$ \\
double star  & $12$ \\
$\gamma$-star     & $6$  \\
double simple & $0$ \\
$A$, $B$, and $C$   & $6$ altogether
\end{tabular}
\end{table}%

Adding up the contributions, we get 
 $6\,{\mathcal S} + 12\, {\mathcal S_2} -6\,\C + 6$.
  Since the 
total number of $\alpha$ is $N$, the total number of $\beta$ is 
$N$, and the total number of $\gamma$ is $N$, we get zero for 
the grand total.   Then ${\mathcal S} + 2\,{\mathcal S_2} = \C-1$. 
Therefore
\begin{eqnarray}
  \C - {\mathcal S}  =  2 \mathcal S_2 + 1 > 0 \label{eq:493}
\end{eqnarray}
Now we consider the graph $\Gamma_a$. Let $f(Q)$ be the out-degree of $Q$ (in $\Gamma_a$) minus the in-degree of $Q$. Also recall that we can have a link ending at $Q$ 
only if $Q$ is a $\pi$-vertex.
\begin{itemize}
    \item At a center $Q$ there are three tiles,
each with an $a$ edge and a $b$ edge at $Q$. Since $3$ is odd,
one of the $a$ edges shares a segment with one of the $b$ edges,
i.e., an $a/b$ edge emanates from $Q$. 
Let $R$ be the 
farthest point from $Q$ along that segment such that $QR$
supports only $a$ tiles on one side, say the ``left'' side.
  If $R$ were a
vertex of a tile on the 
other side, we would have a relation $ja = pb + qc$,
and $p$ would be positive since there is a $b$ edge on 
the ``right'' side of $QR$.  Since by hypothesis, there is 
no such relation, 
$R$ is not a vertex of a tile on the 
other side.  Then $QR$ is a link in $\Gamma_a$ and the outdegree of $Q$ equals $1$. On the the other hand, as a center is a $2\pi$-vertex, the in-degree is zero. Thus, $f(Q) = 1$.
\item  At a star $Q$ on an internal segment $PQ$, six tiles meet, providing six $c$ edges, three $a$
edges, and three $b$ edges.   There could be an incoming link 
at $Q$, if the tile on $PQ$ at $Q$ has its $a$ edge there, 
the tile past $Q$ does not have its $a$ edge on $PQ$ extended,
and the other two $a$ edges are not on the same segment. 
The in-degree of $\Gamma_a$ can never exceed 1, since it is impossible
for two lines of the tiling to cross at $Q$ when a link ends at $Q$.
There might be zero,  one, or more outgoing links from $Q$, as far as we know.%
\footnote{We would be done immediately if we could prove there is at least one outgoing $a$-link here, as assumed in \cite{laczkovich2012}. However, there is a gap in the proof; see Appendix.
}Thus, $f(Q) \geq -1$ for stars.
\item In the cases of $PQ$ when $Q$ is a double star, a $\gamma$-star, or a double simple vertex, $Q$ is a $2\pi$-vertex. Therefore, $PQ$ cannot be a link
of $\Gamma_a$ and the in-degree
is zero. Thus, $f(Q) \geq 0$.
\item At the vertices $A$, $B$, and $C$ the in-degree is zero, since a link
cannot terminate on the boundary. Therefore, $f(A), f(B), f(C) \geq 0$.
\item At simple vertices $Q$, if the in-degree is nonzero, then it must be $1$ (since there is only one direction the $a$-link can come in), but then  Lemma~\ref{lemma:extendlink-a} shows that the outdegree would be $1$.
\end{itemize}

These results are summarized in Table~\ref{table:vertexcount2}.
\begin{table}[ht]
\caption{Outdegree minus indegree by vertex type}
\label{table:vertexcount2}
\centering
\begin{tabular}{ l  c }
 Vertex type  &   Outdegree minus indegree  \\
\hline
simple    & $\ge 0$ \\
center    & $1$  \\
star      & $\ \ \ge -1$ \\
double star  & $\ge 0$ \\
$\gamma$-star     & $\ge 0$  \\
double simple & $\ge 0$ \\
$A$, $B$, and $C$   & $\ge 0$
\end{tabular}
\end{table}

We now sum $f(P)$ over all the vertices. We have
\begin{eqnarray*}
 \sum f(P) \geq {\mathcal C} - {\mathcal S}  && \mbox{\qquad from Table~\ref{table:vertexcount2}}   \\
 \sum f(P) > 0          && \mbox{\qquad by (\ref{eq:493}})
\end{eqnarray*}
On the other hand, $\sum f(P) = 0$ since every link contributes indegree 1 to the vertex at the 
head of the link, and outdegree 1 to the vertex at the tail of the 
link. This is a contradiction. 
\end{proof}

\begin{lemma}
\label{lemma:stubborn}
If a tiling has an $a$-relation and a $b$-relation, then $a/b$ is rational. 
\end{lemma}
\begin{proof}

Assume, for proof by contradiction, 
that $a/b$ is not rational.  Suppose 
$ja = pb + qc$ and $Jb = Pa + Qc$, with nonnegative coefficients and $j, J > 0$.  If either $q = 0$ or $Q = 0$, then $a/b \in \mathbb{Q}$. So we may assume $q > 0$ and $Q > 0$. Therefore we have (solving for $c$ in both):
\begin{eqnarray*}
c = \frac{ja - pb}{q} &=& \frac{Jb-Pa}{Q}\\
(jQ+qP)a &=& (Jq+pQ)b \\
\frac a b &=& \frac {Jq+pQ}{jQ+qP}
\end{eqnarray*}
Therefore $a/b$ is rational.
\end{proof}

\begin{theorem} \label{theorem:rationaltile120}
Let $ABC$ be any triangle,  tiled by 
a tile $(\alpha,\beta, 2\pi/3)$ with non-commensurable angles. Suppose $ABC$ is not similar to the tile. Then the tile has commensurable sides, unless $ABC$ satisfies the following conditions:
\smallskip

(i) Every tile supported by the boundary of $ABC$ has its $c$ edge on the boundary, and 

(ii) $ABC$ is either equilateral or has angles $(\pi/3, 2\alpha,2\beta)$.
\end{theorem} 

\begin{proof}  
Suppose $ABC$ is tiled as in the lemma.  
Unless the tiling satisfies the listed conditions, we must have a $c$-relation $jc = pa + qb$ by Lemma~\ref{lemma:crelation}. By Lemma~\ref{lemma:abrelation}, there is an $a$-relation 
and a $b$-relation.  By Lemma~\ref{lemma:stubborn}, $a/b$ is rational.  Dividing $jc = pa + qb$
by $bj$, we have $c/b = pj(a/b) + q/j$, which is rational, and $c/a = (c/b)(b/a)$, 
which is also rational. Therefore the tile has commensurable sides, as claimed.
\end{proof}

\section{Kites, parallelograms, and the boundary}
\label{sec:invariant}

\subsection{Kites and Parallelograms}

Call a tile \emph{$c$-internal} if its $c$-edge is internal to the tiling. We observe:
\begin{lemma}
\label{lemma:stubborn-matching}
    In a tiling of a convex polygon with noncommensurable sides and $a/b$ rational, there exists a perfect pairing between all the $c$-internal tiles such that each such pair of tiles:
    \begin{enumerate}
        \item have their $c$-edges parallel (in fact, on the same line),  and
        \item the tiles' interiors are on opposite sides with respect to their $c$-edges.
    \end{enumerate}  In other words, the two tiles form a kite or parallelogram up to translation.
\end{lemma}
\begin{proof}
    Consider any $c$-internal tile. Its $c$-edge is part of some
    segment $PQ$ that is not part of a longer segment in the tiling.
    Because the polygon is convex, $PQ$ is an internal segment. $PQ$ induces some linear combination of $a$'s, $b$'s, and $c$'s, by counting the sides of the tiles supported on the two sides of $PQ$.

    We can scale so that $a$, $b$ are integers, though $c$ still must be irrational because the tile has noncommensurable sides. Because $a$ and $b$ are integral and $c$ is not, the number of $c$'s on the two sides must equal. We can arbitrarily pair them and they satisfy the desired conditions. Doing this for all $c$-internal tiles finishes the proof.
\end{proof}

\subsection{Definition of \texorpdfstring{$\zh$} {an invariant}}

We consider all our tilings to be composed of directed segments, where each tile is oriented in 
the counterclockwise direction, supplying a direction to each of its edges.
  Each directed segment belongs to one and only one tile, so the 
segments of the tiling are either boundary segments,  or are composed of different oriented segments
on different sides.  The same segment of the tiling might get different directions from the tiles
on its two sides, as will certainly  happen for the diagonals of a parallelogram or the axis of a kite.
If $PQ$ is a directed segment, we define the {\em angle of $PQ$} to be the directed angle from the positive x-axis to the vector PQ, measured counterclockwise.

\begin{lemma}
    \label{lemma:alignment} Given any tiling of a polygon by $(\alpha, \beta, \gamma=2\pi/3)$ with noncommensurable angles where at least one segment of the tiling coincides with the $x$-axis, every vector aligned with any segment of the tiling has the angle (with respect to the $x$-axis) that can be written uniquely as $(j(\pi/3) + k\alpha),$ where $j \in \{0, 1, \ldots, 5\}$, and $k \in \mathbb{Z}$.
\end{lemma}
\begin{proof}
Any vector positively aligned with the $x$-axis corresponds to angle $0$. Every other angle can be obtained by rotating (including clockwise) some integer multiples of $\alpha$, $\beta = \pi/3 - \alpha$, or $\gamma = 2\pi/3$, along with reversing directions, which is the same as rotating by $\pi$.  Since $\alpha$ is not a rational multiple of $\pi$, the coefficient $k$ is uniquely determined, which then determines $j$ up to $\mathbb{Z}/6\mathbb{Z}$. 
\end{proof}

We now define the invariant $\zh$, introduced by Laczkovich on p.~351 of \cite{laczkovich2012}; see 
also the bottom of p.~365.  
Let $\A$ be the set of numbers of the form $j (\pi/3) + k\alpha$.
 Define $\zh : \A \to \R$ by 
$$ \zh (j (\pi/3) + k\alpha):= (-1)^j.$$
We use the same letter $\zh$ for a mapping from directed segments $PQ$ to $\A$ defined
by $\zh (PQ) = \vert PQ \vert \zh (\theta)$, where $\theta$ is the angle of $PQ$ as defined above.
If $PQR$ is a tile, with vertices listed in counterclockwise order, so that $PQ$, $QR$, and $RP$
can be considered as directed segments,  
we define 
$$ \zh (PQR) := \zh (PQ) + \zh (QR) + \zh (RP).$$
  Finally we define $\zh$ on a tiling 
(a union of tiles $T_i$) to be $\sum \zh (T_i)$.  

\begin{lemma} \label{lemma:fmultiplicative}
For $\theta$ and $\psi$ in $\A$, we have $\zh (\theta+\psi) = \zh (\theta) \zh (\psi).$
\end{lemma}

\begin{proof}
Let $\theta = j(\pi/3) + k \alpha$ and $\psi = p(\pi/3) + q \alpha$.  Then
\begin{eqnarray*}
\zh (\theta+\psi) &=& \zh ((j+p) (\pi/3) + (k+q) \alpha) \\
               &=&  (-1)^{j+p}  \\
               &=&  (-1)^j (-1)^p \\
               &=&  \zh (\theta) \zh (\psi).
\end{eqnarray*}
\end{proof}

\begin{lemma} \label{lemma:boundary} Let $T$ be any tiling by a tile $(\alpha,\beta,2\pi/3)$ with non-commensurable angles.    Then $\zh (T)$ is equal to 
the sum of $\zh (PQ)$ over the directed boundary segments $PQ$ of $T$ in counterclockwise order.
\end{lemma}

\begin{proof}
$\zh (T)$ is defined as the sum of $\zh (PQ)$ over directed edges of the tiles in $T$.
Group the terms of this sum not by tile but by line of the tiling.  On each maximal line $L$ (i.e., $L$
terminates in points where it cannot be extended as a line of the tiling), the directed edges on one side 
are all oriented oppositely to the directed edges on the other side, and they all have the same angle,
since they lie on the same line.  We have $\zh (PQ) = -\zh (QP)$ by definition of $\zh$, since the angles of 
$PQ$ and $QP$ differ by $\pi$.
Hence the sum of $\zh (PQ)$ over all directed segments on both sides of $L$
is zero.  Since $\zh (T)$ is the sum over the interior segments plus the sum over the boundary segments,
and the sum over internal segments is zero, $\zh (T)$ is equal to the sum over the boundary segments.
\end{proof}  

\subsection{A worked example}
We intend to use $\zh(T)$ to show that certain 
tilings are impossible.  Those proofs will 
involve calculations of $\zh$ on hypothetical
tilings that do not actually exist--indeed their
non-existence is the point!  But to fix the 
ideas, in this section we will calculate
$\zh(T)$ on a specific tiling that does 
exist.  The reader who thinks the definition
of $\zh(T)$  does not need clarification by 
example may skip this section, as it plays
no role in our proof.

We calculate $\zh(T)$ for the smallest known tiling of an equilateral
triangle by a tile with a $2\pi/3$ angle and $\alpha/\pi$ irrational.  
This tiling was found in January 2024 by Bryce Herdt.   See Fig.~\ref{figure:1215}.

\begin{figure}[ht]
\centering
\includegraphics[width=0.85\linewidth]{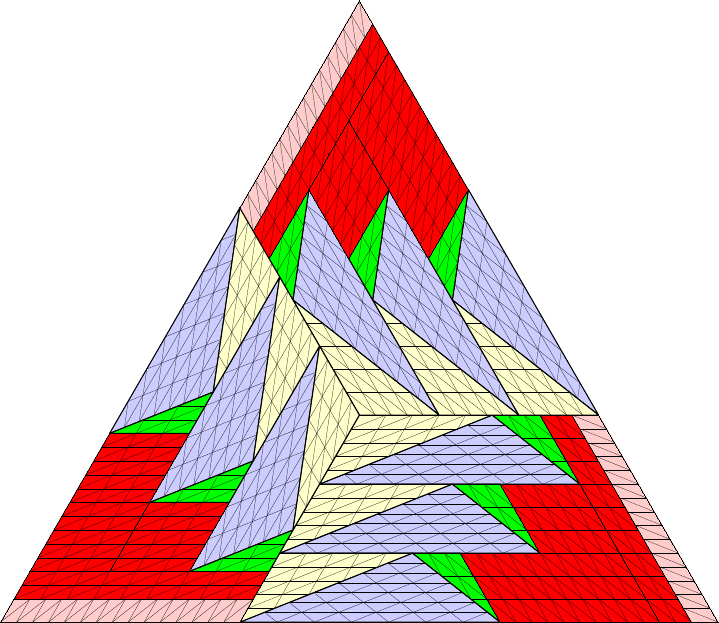}
\caption{$N=1215$.  The tile is $(3,5,7)$ and $\gamma = 2\pi/3$.}
\label{figure:1215}
\end{figure}

Since the tile has commensurable sides, Lemma~\ref{lemma:stubborn-matching} does not apply and we do not {\em a priori} get a matching of the $c$-internal segments, which is what makes this tiling possible. 
We want to demonstrate  by direct computation that in this example $\zh$ of the boundary is equal to $\zh$ of the tiling,
 Each side of the equilateral triangle is $27a + b + 7c$,
as you can count in the figure.  Since $(a,b,c) = (3,5,7)$,  that comes to $L = 135$.  Then 
twice the area is
$L^2 \sin (\pi/3) = 135^2 \sqrt 3 /2.$ 
Twice the area of each tile is $ab \sin(2\pi/3) = 3 \cdot 5 \sqrt{3}/2$.
With $N = 1215$ we should have $ 1215 \cdot 3 \cdot 5 = 135^2$.  Indeed, both sides are 18225.  

Then according to Lemma~\ref{lemma:boundary}, 
$\zh(T)$ is given by
\begin{eqnarray}
\zh(\partial ABC) = 135 (1 - (-1)^1 +  (-1)^2) =3\cdot 135 \label{eq:968}
\end{eqnarray}
We want to calculate $\zh(T)$ directly from the definition
and verify that we get this answer.

The red and pink parallelograms contribute zero.
The rest of $T$ is composed of nine trapezoids $T_1,\ldots, T_9$, each containing
yellow, light blue, and green tiles.   Let $T_1$ be the one 
at the bottom of the figure.    
Consider a single trapezoid $T_1$ with the longest side aligned to the $x$-axis, such as the one at the bottom of the tiling. The angles of the sides, starting with the bottom, are $(0, 2\pi/3, \pi, 4\pi/3)$,
corresponding to factors of $(-1)^0, (-1)^2, (-1)^3, (-1)^4)$.  
We can compute by counting tiles in the figure that 
$$  \zh(\partial T_1) = 49(-1)^0 + 15(-1)^2 + 34(-1)^3 + 15 (-1)^4 = 49+15-34+15 = 45$$
If we accept for the moment that $\zh(T_1) = \zh(\partial T_1)$,
then summing the nine copies (with triplets of the copies rotated by $2\pi/3)$ gives 
$$45 \cdot 9 = 3 \cdot 135$$ as desired.

A {\em quadratic tiling} is a tiling using $n^2$ tiles
similar to the tiled triangle, formed by lines parallel
to the sides of the tiled triangle.  For example,
the nine light blue triangles in Fig.~\ref{figure:1215}. That Lemma~\ref{lemma:boundary} works for a quadratic tiling 
can be seen visually.  For example,
the non-parallelogram part of the blue quadratic
tiling $Q$ is the seven tiles supported by the lower boundary $AB$ of $ABC$.
From the translation-invariance of $\zh$, one sees that
$\zh$ of  those tiles is equal to
$$\zh(Q) =  7c + 7a(-1)^2  - 7b(-1),$$
which is $\zh(\partial Q)$ (and incidentally comes to $7c + 7a + 7b$).
Similarly Lemma~\ref{lemma:boundary} works for the green and yellow 
quadratic tilings.  Since the touching edges of the green and blue 
triangles are oppositely oriented, their
 terms  in $\zh(\partial Q)$ will cancel.  Similarly
 for the touching edges of yellow and blue.
 Then only the 
terms for the boundary of the trapezoid $T_1$ are left.
That completes the exercise.

\section {Applications of the invariant}
\label{sec:proof-main}
In this section we apply the invariant $\zh$
to prove the non-existence of certain tilings,
because if they existed,  the sum of $\zh(P)$
over tiles $P$ would not equal the value of 
$\zh$ on the boundary.

\begin{lemma} \label{lemma:kites}
Let $T$ be any tiling using the tile $(\alpha,\beta,2\pi/3)$ and non-commensurable angles. 
Then parallelograms and kites contribute $0$ to $\zh (T)$.
\end{lemma}
\noindent{\em Remark}.
Because $\zh(T)$ is translation-invariant, this lemma applies also to
``virtual'' kites and triangles produced by different $c$ edges on opposite sides
of the same line as in Lemma~\ref{lemma:stubborn-matching}.
\begin{proof}
First, let us compute the angles of parallelograms and kites:
\begin{itemize}
    \item Let $PQRS$ be a counterclockwise parallelogram
with $PQ$ having angle $\theta$ and angle $PQR = 2\pi/3$.  The angles of the four segments are
\begin{eqnarray}
(\theta, \theta+ \pi/3, \theta -\pi, \theta -2\pi/3) \label{eq:3317}   
\end{eqnarray}
\item Let $PQRS$ be a counterclockwise kite with $PQ$ have angle $\theta$
and angle $PQR = 2\pi/3$ and angle $SPQ = 2\alpha$.  The angles of the four segments are
\begin{eqnarray}(\theta, \theta + \pi/3, \theta + \pi/3 +(\pi- 2 \beta), \theta + 2\alpha + \pi)  \nonumber\\
= (\theta, \theta + \pi/3, \theta + 2\pi/3 + 2 \alpha, \theta + \pi + 2 \alpha ).  \label{eq:3324}
\end{eqnarray}
\end{itemize}

The opposite sides of a parallelogram have opposite angles. Their contributions to $\zh$ 
thus differ by $(-1)^3 = -1$.  More formally, if $PQRS$ is a parallelogram, the angles are given by (\ref{eq:3317}); by Lemma~\ref{lemma:fmultiplicative} we can factor out $\zh (\theta)$.   Assume $\vert PQ \vert = b$.
Then 
$$ \zh (PQRS) = \zh (\theta)(b + a(-1) + b(-1)^{-3} + a(-1)^{-2}) = b-a-b+a = 0,$$
as claimed.
    
    For a kite: without loss of generality, a kite has counterclockwise vertices $PQRS$ with $SPQ = 2\alpha$ and $QRS = 2\beta$. Let the angle of $PQ$ be $\theta$, so $\zh (PQ) = b\zh (\theta)$ and the other angles are
given in counterclockwise order by (\ref{eq:3324}), namely 
\begin{eqnarray*}
= (\theta, \theta + \pi/3, \theta + 2\pi/3 + 2 \alpha, \theta + \pi + 2 \alpha ).  
\end{eqnarray*}
Therefore
$$\zh (PQRS) = \zh (\theta)( b + a(-1) + a(-1)^2+ b(-1)^3) = 0,$$
as claimed.
\end{proof}

\begin{theorem}
\label{theorem:equilateralrational}
Suppose the equilateral triangle $ABC$ is tiled by $(\alpha, \beta, 2\pi/3)$ with noncommensurable angles. Then the tile has commensurable sides.
\end{theorem}

\begin{proof}
Suppose $ABC$ is tiled as in the lemma.  Suppose, for proof by contradiction,
that the tiling has non-commensurable sides.  Let $(a,b,c)$ be the sides of the tile opposite
to $(\alpha,\beta,\gamma)$ respectively. By Lemma~\ref{lemma:stubborn}, we can assume $a$, $b$ are integers and $c$ is irrational.  By Theorem~\ref{theorem:rationaltile120},
each side of $ABC$ supports only tiles with the $c$ edge on the boundary of $ABC$.
We align so that the bottom edge of the equilateral triangle is the $x$-axis, which we assume is $AB$,
with vertex $C$ in the upper half plane so $ABC$ is counterclockwise order.

Now, we use a trick to write a slightly augmented version of $\zh(T)$ as a sum of contributions only from kites and parallelograms. Around the outside of triangle 
$ABC$, we add new tiles, all oriented the same way, say with sides in the order $abc$ as we 
traverse $ABC$ counterclockwise.  We call these the ``sawtooth'' tiles.  See Fig.~\ref{figure:sawtooth}. 

\begin{figure}[ht]
\centering
\includegraphics[width=0.85\linewidth]{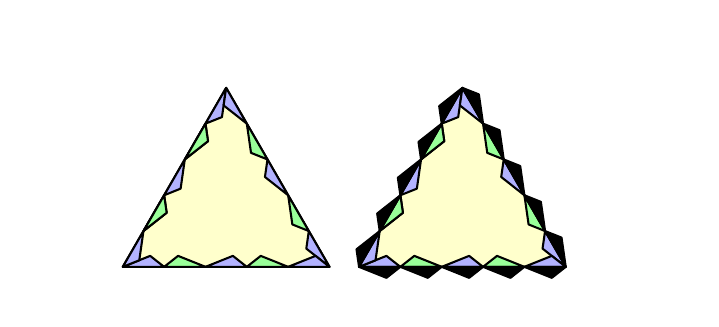}
\caption{Before and after adding sawtooth tiles}
\label{figure:sawtooth}
\end{figure}
Each sawtooth tile matches with 
a boundary tile to make either a kite or a parallelogram, depending on the orientation of that boundary tile. By Lemma~\ref{lemma:stubborn-matching}, all the tiles not supported by the boundary, being $c$-internal, are paired into kites and parallelograms, and we have just paired up the boundary-supported tiles (which are all not $c$-internal) into kites and parallelograms as well. Let $T^\prime$ be the tiling of $ABC$ with the sawtooth tiles added.  Then
$\zh(\partial T^\prime)$  is the sum of contributions from kites and parallegrams, which are zero.   
Therefore
\begin{eqnarray}
\zh(\partial T^\prime) = 0   \label{eq:794}
\end{eqnarray}

Let $PQR$ be a sawtooth tile supported by $AB$,  in which $PQ$ has length $a$
and angle $$-\beta = 2\pi-\beta = 2\pi  - (\pi/3-\alpha) = 5\pi/3 + \alpha,$$ 
and $RP$ has length $b$ and angle $5(\pi/3) + \alpha + \pi/3 = \alpha$.  Then
\begin{eqnarray*}
\zh(QR) &=& (-1)^5 a = -a \\
\zh(RP) &=& b \\
\zh (QR) + \zh (RP) &=&  b-a 
\end{eqnarray*}
Note that we do not include $PQ$, which is now part of a kite or parallelogram, and does not lie
on the boundary of $T^\prime$. 
We have $b \neq a$, since $b/a = \sin \beta / \sin \alpha$, and $\alpha + \beta = \pi/3$, so if $a=b$
we would have $\alpha = \beta$, contradicting the hypothesis of incommensurable angles. 

Since there are only $c$-edges on the boundary, the length of each side is $Xc$, where $X$ is the 
number of tiles in $T$ supported by each side of the boundary, which also equals the number of sawtooth tiles on each side.
The sawtooth tiles on the other two sides of $ABC$ have angles like $PQR$, but rotated by $2\pi/3$ and $4\pi/3$.
More precisely, the boundary segments of those sawtooth tiles have angles rotated from the boundary segments
of sawtooth tiles on $AB$.
So they give the same contributions, multiplied by $(-1)^2$ and $(-1)^4=1$ respectively. That is,
\begin{eqnarray*}
\zh(\partial T^\prime) = 3X(b-a)
\end{eqnarray*}
But that contradicts (\ref{eq:794}), since $b \neq a$.
\end{proof} 

\begin{theorem}
\label{theorem:lastcase}
Suppose the triangle $ABC$ with angles $(2\alpha, \pi/3, 2\beta)$ is tiled by  $(\alpha, \beta, 2\pi/3)$ with noncommensurable angles. Then the tile has commensurable sides. 
\end{theorem}

\begin{proof} 
Let $ABC$ be tiled as given, with $(2\alpha, \pi/3, 2\beta)$ at $A,B,C$ respectively. Suppose, for proof by contradiction,
that the tiling does not have commensurable sides.  Let $(a,b,c)$ be the sides of the tile opposite
to $(\alpha,\beta,\gamma)$ respectively. By Lemma~\ref{lemma:stubborn}, we can assume $a$ and $b$ are integral and $c$ is irrational. By Theorem~\ref{theorem:rationaltile120},
each side of $ABC$ supports only tiles with the $c$ edge on the boundary of $ABC$.
As in the proof for equilateral $ABC$, we add sawtooth tiles
around the outside of the boundary of $ABC$, matching the $c$ edges.
The tiling $T^\prime$  of this extended polygon is composed of kites and parallelograms.  
Therefore $\zh(T^\prime) = 0$.  
Therefore 
\begin{eqnarray}
\zh(\partial T^\prime) = 0 \label{eq:836}
\end{eqnarray}
Let $kc$, $lc$, and $mc$ be the lengths of $AB$, $BC$, and $CA$ respectively. Then we must have $k, l, m > 0$ integers and 
 $k + \ell + m $ sawtooth tiles.  As before, the sawtooth tiles on the lower boundary 
 each contribute $b-a$ to $\zh(\partial T^\prime)$.  But now, the rotation factors are different.
 The angle of $AB$ is 0; the angle of $BC$ is $2\pi/3$, giving a rotation factor of $(-1)^2 = 1$.
 The angle of $CA$ is 
 $$2\pi/3 + (\pi-2\beta) = 2(\pi/3-\beta) + \pi = 2\alpha + \pi = 3(\pi/3) + 2\alpha.$$
 So $\zh$ of that angle is $(-1)^3 = -1$.  Hence
 $$ \zh(\partial T^\prime) = (b-a) (k + \ell - m).$$
 Since the lengths of the sides of $ABC$ are $kc, \ell c $, and $mc$,  we have
 $kc + \ell c > mc$.  Hence $k+\ell-m \neq 0$, contradicting (\ref{eq:836}), since $b \neq a$.
 \end{proof}

\section{Conclusion}
\label{sec:conclusion}
At the outset, we stated Theorem~\ref{theorem:main} as
\medskip

\noindent
\textbf{Theorem~\ref{theorem:main}} 
{\em Suppose a triangle $ABC$ is tiled by a tile $R$ with  sides $(a,b,c)$ and noncommensurable angles $(\alpha,\beta,\gamma = 2\pi/3)$. If $ABC$ is not similar to $R$, 
then $R$ has commensurable sides.}

\begin{proof} Theorem~\ref{theorem:rationaltile120} shows that any counterexample must have all $c$ edges
on the boundary, and $ABC$ must be equilateral or have angles $(2\alpha, 2\beta, \pi/3)$.  
Theorem~\ref{theorem:equilateralrational} and Theorem~\ref{theorem:lastcase} handle those two cases.
\end{proof}

Now recall the statement of Theorem~\ref{theorem:main-2}:
\medskip

\noindent
\textbf{Theorem~\ref{theorem:main-2}}
{\em Let triangle $ABC$ be tiled by a tile $R$ such that
    \begin{itemize}
        \item $R$ is not similar to ABC;
        \item $R$ is not a right triangle;
        \item $R$ does not have commensurable angles.
    \end{itemize}
    Then $R$ must have commensurable sides. }
\medskip  

\noindent
{\em Remark}.  The exceptions are necessary: any triangle $ABC$ can be tiled by a triangle similar to itself, and we can put together any two copies of an arbitrary right triangle to obtain an isosceles triangle, so we cannot conclude anything about the commensurability of the sides of the tile in these two cases. 
\medskip

\begin{proof}
In Table~\ref{table:citations}, we present cases from Table~\ref{table:laczkovich2} that meet  the 
hypotheses of Theorem~\ref{theorem:main-2}.  We add a column with
references for the proof of Theorem~\ref{theorem:main-2}
in those cases.
\begin{table}[ht]
\caption{Citations for the proof of Theorem~\ref{theorem:main-2}}
\label{table:citations}
\centering
\setlength{\extrarowheight}{.2em}
\begin{tabular}{lcl}
$ABC$ \quad &  The tile  & Citation    \\
\hline
equilateral              & $\alpha = \pi/3$ & \cite{laczkovich2012}, Lemma~3.2 and Theorem~3.3 \\
equilateral              & $\gamma = 2\pi/3$  & Theorem~\ref{theorem:main} of this paper \\
$(2\alpha,2\beta,\pi/3)$  & $\gamma = 2\pi/3$ & Theorem~\ref{theorem:main} of this paper \\
$(\alpha,\alpha,\pi-2\alpha)$  & $\gamma = 2\alpha$ & Theorem~10.5 of \cite{beeson-isosceles}\\
$(2\alpha,\beta,\alpha+\beta)$ & $3\alpha + 2\beta = \pi$  & Theorem~3 of \cite{beeson-triquadratics}\\
$(2\alpha,\alpha,2\beta)$ &$3\alpha + 2\beta = \pi$  &  Theorem~3 of \cite{beeson-triquadratics}\\
isosceles &$3\alpha + 2\beta = \pi$  &  Theorem~3 of \cite{beeson-triquadratics}
\end{tabular}
\end{table}
 \end{proof}
 
\appendix

\section{The gap in \texorpdfstring{\cite{laczkovich2012}}{[4]}}

The gap in question occurs at lines 13-14 on page~364 (the last lines of the second paragraph of the proof
of Lemma~7.1).
Fig.~\ref{figure:sevenpointone} illustrates the situation.  

\begin{figure} [ht]
\caption{Lemma 7.1}
\label{figure:sevenpointone}
\centering
\includegraphics[width=0.5\linewidth]{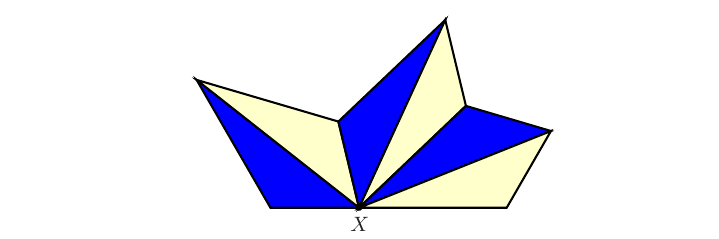}
\end{figure}

The six triangles $T_1,\ldots,T_6$ with vertices at $X$ are shown.
Their angles at $X$ are $\beta$, $\beta$, $\beta$, $\alpha$, $\alpha$, $\alpha$.
The out-degree of $\Gamma_a$ at $X$ is zero.
The in-degree at $X$ is 1, if $X$ lies on an internal edge of the tile below the figure.
That is contrary to assertion (i) of Lemma~7.1, and the assertion of the last two lines of the second paragraph
of the proof fails to hold in Fig.~1.  Moreover, the first lines of the next paragraph assert,
``The argument above shows that if the equation at a vertex $X$ is $3\alpha + 3\beta = \pi$, then an edge of $\Gamma_a$ starts from $X$.''  But that is patently false in the figure. 

This problem propagates to the $2\pi/3$ case of item~(iv) of Theorem~2.1 of \cite{laczkovich2012}, 
which is Theorem~\ref{theorem:equilateralrational} of this paper. (The hypothesis that the tiling is 
regular must be fulfilled in the case of equilateral $ABC$ for tiles with incommensurable angles.)
In this paper, we have supplied a correct proof.  Our theorem~\ref{theorem:lastcase} is not
implied by the results of \cite{laczkovich2012}, since the tiling is not regular.

\bibliographystyle{plain-annote}
%\bibliography{TriangleTiling}

\end{document}